\documentclass[11pt]{article}
\usepackage{mathrsfs}
\usepackage{mathtools}
\usepackage{amssymb}
\usepackage{amsmath}
\usepackage{amsthm}
\usepackage{amscd}
\usepackage{commath}
\usepackage{enumerate}
\usepackage{stmaryrd}
\usepackage{enumitem}
\usepackage{hyperref}
\usepackage{thmtools}
\usepackage{listings}
\usepackage{tikz}
\usepackage{tikz-cd}
\usepackage{comment}
\usepackage{lipsum}
\usetikzlibrary{matrix,arrows,decorations.pathmorphing}
\usepackage{bm}
\theoremstyle{definition} 
\newtheorem{definition}{Definition}

\theoremstyle{plain} 
\newtheorem{theorem}{Theorem}[section]
\newtheorem{proposition}[theorem]{Proposition}
\newtheorem{lemma}[theorem]{Lemma}
\newtheorem{corollary}[theorem]{Corollary}
\newtheorem{conjecture}{Conjecture}

\theoremstyle{remark} 

\newtheorem{remark}{Remark}

\DeclareMathOperator{\Spec}{Spec}

\DeclareMathOperator{\Hom}{Hom}

\DeclareMathOperator{\End}{End}
\DeclareMathOperator{\Frob}{Frob}
\DeclareMathOperator{\Hilb}{\cH}
\DeclareMathOperator{\Spf}{Spf}

\DeclareMathOperator{\Split}{Split}

\newcommand{\ZZ}{\mathbb{Z}}

\newcommand{\F}{\mathbb{F}}
\newcommand{\Fpbar}{\overline{\mathbb{F}}_p}
\newcommand{\Z}{\mathbb{Z}}
\newcommand{\cO}{\mathcal{O}}
\newcommand{\I}{\mathscr{I}}
\newcommand{\IF}{\widetilde{\mathscr{I}}}
\newcommand{\cS}{\mathcal{S}}

\newcommand{\pdiv}{\mathcal{G}}
\newcommand{\abvar}{\mathcal{A}}

\title{{Just-likely intersections on Hilbert modular surfaces}} 

\author{{Asvin G.\footnote{email: gasvinseeker94@gmail.com}, Qiao He, Ananth N. Shankar\footnote{All authors affiliated with the University of Wisconsin-Madison.}}}

\date{} 

\begin{document}





\maketitle 
 
\setcounter{tocdepth}{2} 


\begin{abstract}
In this paper, we prove an intersection-theoretic result pertaining to curves in certain Hilbert modular surfaces in positive characteristic $p$. Specifically, let $C,D$ be two proper curves inside a mod $p$ Hilbert modular surface associated to a real quadratic field split at $p$. Suppose that the curves are generically ordinary, and that at least one of them is ample. Then, the set of points in $(x,y) \in C\times D$ with abelian surfaces parameterized by $x$ and $y$ isogenous to each other is Zariski dense in $C\times D$, thereby proving a case of a just-likely intersection conjecture. We also compute the change in Faltings height under appropriate $p$-power isogenies of abelian surfaces with real multiplication over characteristic $p$ global fields. 
\end{abstract}





\newcommand{\cH}{\mathcal{H}}

\section{Introduction}

The study of intersection theory in the context of Shimura varieties has produced several breakthroughs in arithmetic. For instance, Gross-Zagier \cite{GZ} prove their celebrated theorem on ranks of elliptic curves by extensively studying the arithmetic intersection theory of Heegner points on modular curves. This program has been generalized in various
directions. For example, it has been extended to Shimura curves (\cite{KRY}, \cite{YZZ}), higher dimensional unitary Shimura varieties (\cite{LL2020}), and function field analogue of Shimura curves (\cite{YZ}). Arithmetic intersection theory has been crucial to the development of this program. 

In a different direction, the proof of the Andr\'e-Oort conjecture heavily relies on the Average Colmez conjecture (interesting in its own right), whose proof (\cite{AGHMP}, \cite{YuanZhang}) relies on the Arakelov intersection theory of special cycles in orthogonal Shimura varieties. 

Most relevant to our paper is the work of Charles \cite{Ch} proving that two elliptic curves over a number field are isogenous modulo infinitely many primes. This is achieved by studying the arithmetic intersection theory of the modular curve by an approach inspired by work of Chai-Oort, who prove the analogous result over function fields \cite[Proposition 7.3]{CO06}. This work has been generalized to the splitting of abelian surfaces over global fields \cite{ST}, \cite{SSTT}, \cite{MST}, (where the authors prove that abelian surfaces over global fields are isogenous to a product of elliptic curves modulo infinitely many primes) and to Picard-rank jumping results \cite{SSTT} \cite{MST1}. These theorems are obtained by proving that any fixed arithmetic curve intersects the Hecke translates of an appropriate ``special divisor'' at infinitely many points. These special divisors solve moduli problems, and parameterize abelian varieties with extra endomorphisms (the interested reader may look at \cite{AGHMP} or \cite[Section 2.3-2.5]{SSTT} for precise definitions). This moduli-theoretic interpretation of special divisors is crucial to understanding and computing local and global intersections.

Despite the proofs relying heavily on this moduli theory of special divisors, these problems can be naturally phrased in the setting of arbitrary subvarieties of a Shimura variety of Hodge type. Specifically, let $(G,X)$ denote a Shimura datum of Hodge type with reflex field $E$. Let $S/E$ denote the Shimura variety of Hodge type associated to $(G,X)$. We also work with a fixed Hodge embedding of $S$ into the moduli space of polarized abelian varieties. We assume the level structure $\mathsf{K} \subset G(\mathbb{A}^f)$ is neat, so that $S$ carries a universal abelian scheme.  Let $\cS/\cO_E[1/N]$ denote the integral canonical model of $S$ as constructed by Kisin in \cite{Kisinintegral} where $N$ is a sufficiently large integer. We refer the interested reader to \emph{loc. cit.} for the definition of integral canonical models and for a precise description of $N$ in terms of $G$ and the level structure $\mathsf{K}$. Let $C,D\subset \cS$ be subschemes having complementary dimension. How does the set of points on $C$ isogenous\footnote{The universal abelian scheme on $\cS$ defines for us the notion of isogeny.} to some point of $D$ distribute in $C$? In a recent AIM workshop, the participants developed a framework that makes the following prediction.

\begin{conjecture}\label{AIMconjecture}
Let $\cS$ be as above, and let $\cS_{\F_q}$ denote the fiber of $\cS$ at some prime $\mathfrak{p} \subset \cO_E[1/N]$. Suppose that $C,D\subset \cS_{\F_q}$ are generically ordinary subvarieties having complementary dimension. Then, the set of points in $C$ isogenous to some point of $D$ is Zariski-dense in $C$. Further, the subset of $C\times D$ consisting of pairs of isogenous points is dense in $C\times D$.
\end{conjecture}

This conjecture is inspired both by the work in \cite{Ch, CO06, ST, SSTT,MST} as well as the Hecke orbit conjecture. There is also an arithmetic (i.e., number-field) analogue of this conjecture. These conjectures have several applications. For instance, they would imply that any abelian fourfold over a global field should be isogenous to a Jacobian modulo infinitely many primes! This is all the more striking as the existence of abelian fourfolds over global fields \emph{not} isogenous to Jacobians has been established in \cite{AnanthJacob} and \cite{JacobJacobian}. The characteristic-zero analogue of Conjecture \ref{AIMconjecture} has been proven (\cite[Theorem 1.22]{TT}). 

In \cite{Asvin}, the first-named author proves this conjecture in the setting of $X(1)^n$. In this paper, we establish this conjecture for Hilbert modular surfaces (with some mild constraints).
\begin{theorem}\label{intromain}
Let $F$ denote a real quadratic field, and let $p$ denote a rational prime that splits in $F$. Let $\cH$ denote the mod $p$ Hilbert modular surface associated to $\cO_F$, and let $C,D\subset \cH$ denote two generically ordinary, proper curves in $\cH$, at least one of which is ample. Then Conjecture \ref{AIMconjecture} holds for $C,D \subset \cH$. \end{theorem}
We expect this result to hold for curves that are not necessarily proper or ample. We are hopeful that an appropriate prime-to-$p$ Hecke translate of a fixed curve in $\cH$ should be ample (but to our knowledge, this is not yet known). Such a result would immediately reduce the non-ample case to the ample case.

\subsection{Outline of proof}
Our proof is inspired by Chai-Oort's work. Roughly speaking, we consider a sequence of Hecke operators $T_n$, and consider the intersections $(T_n(C). D)$. In order to prove our theorem, it suffices to prove that the local contributions of any finite set of points $x_1 \hdots x_m \in D$ to the intersection $(T_n(C).D)$ is smaller than the global intersection number $(T_n(C).D)$ for some $n$. However, we encounter several additional difficulties that aren't present in their work.

Firstly, their ambient space is a product of curves, and their proof makes crucial use of the product structure. In the hitherto solved cases where the ambient space is \emph{not} a product variety, either $C$ or $D$ is special and the local intersection numbers at any fixed point is bounded using the moduli interpretation of the special divisor (see \cite[Lemma 7.2]{SSTT}, and \cite[Lemma 7.2.1]{MST}). We overcome this first difficulty by constructing a \emph{local} product structure at every closed point (see Proposition \ref{prop: local co-ords}) which is adapted to a very specific set of Hecke operators that exist only in positive characteristic. In fact, these Hecke operators endow $\cH$ with a partial Frobenius structure (see \ref{defn: r(x)}). This construction allows us to control the local intersection multiplicites (see Theorem \ref{thm: local intersection, moving curve}). 

The second difficulty is in estimating the global intersection number. The previously known cases (in the setting of special divisors) use earlier results establishing the modularity of sequences of special divisors (\cite{Bor99}) to compute the global intersection numbers. In our setting, we have no such modularity results and we instead compute the intersections of $T_n(C)$ with the non-ordinary locus of $\cH$, and then use the ampleness of $D$ to control $(T_n(C).D)$. While doing so, we establish a result (Theorem \ref{thm: change of heights formula}) pertaining to the change of Faltings height under $p$-power isogeny which is of independent interest. 

\subsection{Change of Faltings height under $p$-power isogeny}

The Faltings height of an abelian variety over a number field is defined to be the Arakelov degree of the Hodge bundle. The change of the Faltings height of an Abelian variety under isogeny was a crucial ingredient in the proof of the Mordell Conjecture by Faltings. 

In the function field case, the definition of Faltings height for a proper curve is even simpler. Suppose $\abvar/C$ is a family of abelian varieties over a proper curve $C$. Then we define the Faltings height of $\abvar/C$ to just be the degree of the Hodge bundle. In our setting, this height only depends on the image of $C$ in $\cH$. The prime-to-$p$ Hecke operators on $\cH$ are \'etale, and this makes it easy to compute the change in Faltings height under prime-to-$p$ isogeny. The non-\'etaleness of $p$-power Hecke operators makes computing the change-in-height a formidable prospect. Further, $p$-power Hecke correspondences aren't even separable! 

Indeed, the only prior work pertaining to this that we are aware of is \cite{griffon2020isogenies}. They consider an isogeny $\varphi: E_1 \to E_2$ between two elliptic curves over $k(C)$ for $C$ a curve over $\mathbb F_q$ and \cite[Theorem A]{griffon2020isogenies} proves that
\[ h(j(E_2)) = \frac{\deg_{\mathrm{ins}}(\varphi)}{\deg_{\mathrm{ins}}(\hat\varphi)}h(j(E_1))\]
where $h$ denotes the Weil height on $X(1)$ and $\deg_{\mathrm{ins}}(\varphi)$ denotes the inseparable degree of $\varphi$. The proof there proceeds by an identification of the Weil height with an intersection number: since the Picard group of $X(1)$ is simply $\mathbb Z$, and the only $p$-power Hecke operators (mod $p$) are iterates of Frobenius and Verschiebung, this lets one calculate the effect of a Hecke translation on the Picard group simply as multiplication by an integer.

In our case, the Picard group of $\Hilb$ is complicated and the action of $p$-power Hecke correspondences (which are not \'etale!) on the Picard group is not a-priori easy to compute. Nevertheless we fully compute the change in Faltings  height of an abelian surface with real multiplication by $\cO_F$ under $p$-power isogenies that respect the $\cO_F$-action. 

Suppose $C$ is a proper curve and $\abvar/C$ is a generically ordinary abelian surface with endomorphisms by $\cO_F$ (with $F$ a real quadratic field as in Theorem \ref{intromain}) so that the image of $C$ under the corresponding map $C \to \cH$ is an ample divisor of $\cH$. Suppose moreover that $\abvar_n/C$ is another abelian surface, also having endomorphisms by $\cO_F$, isogenous to $\abvar/C$ by a purely inseparable isogeny $\varphi: \abvar' \to \abvar$ which respects the $\cO_F$-actions, whose kernel is cyclic of size $p^n$ (indeed, this is the hardest case to treat, and the general case can be deduced from this case). Let $h_F(\abvar/C)$ denote the Faltings height of the abelian variety $\abvar/C$, which is defined to be the degree of the Hodge bundle $\omega$ of $\cH$ restricted to $C$. Then, the height of $\abvar_n/C$ grows exponentially in $n$.

In fact, we prove an exact formula for the change of height under inseparable isogenies (Theorem \ref{thm: change of heights formula}) but for simplicity of notation, we will be content with the version above for the introduction. We prove this result by comparing the Hodge bundle to the class of the Hasse invariant, and using our computation of the intersection of Hecke orbits of $C$ with the non-ordinary locus, and the product structure that we construct.

\subsection*{Acknowledgements}
We're very grateful to Jordan Ellenberg, Eyal Goren, Ruofan Jiang, Yunqing Tang, Salim Tayou, Yifan Wei and Tonghai Yang for useful discussions. We are also grateful to the AIM workshop on Arithmetic Intersection Theory on Shimura Varieties, in particular Shou-Wu Zhang's group, where the broader framework extending Theorem \ref{intromain} was discussed and formalized. We thank the referees for their careful readings and helpful suggestions. 

Q.H. is partially supported by a graduate school grant of  UW-Madison. A.N.S. is partially supported by the NSF grant DMS-2100436.

\section{Background material}\label{sec: background}

We provide some background (and standardize notation) on Hilbert modular surfaces in this section, define the notion of a partial Frobenius structure and prove a theorem about curves parametrizing abelian varieties without extra endomorphisms.

Let $F$ be a real quadratic field, $\cO_F$ its ring of integers, $p=\mathfrak{P}_1\mathfrak{P}_2$ a prime that splits in $\cO_F$ and $\mathfrak{a}$ a fractional ideal of $\cO_F$. One can define a moduli space parametrizing Abelian surfaces with endomorphisms by $\cO_F$:

\begin{definition}
Let $\cH_{n}^{\mathfrak a}$ be the moduli functor that associates to a $\mathbb F_p$-scheme $S$ the groupoid of tuples $(A,\iota,\lambda,\eta)$ where:
\begin{enumerate}
    \item $A\to S$ is an abelian surface;
    \item $\iota: \cO_F\to \operatorname{End}_S(A)$ is a ring homomorphism ;
    \item $\lambda : \mathfrak{a}\to \operatorname{Hom}^{\operatorname{Sym}}_{\cO_F}(A,A^{\vee})$ is an $\cO_F$-linear homomorphism such that $\lambda (a)$ is an $\cO_F$-linear polarization of $A$ for every totally positive $a\in \mathfrak{a}$, and the homomorphism $A\otimes_{\cO_F} \mathfrak{a} \stackrel{\sim}{\to}A^{\vee}$ induced by $\lambda$ is an isomorphism of abelian surfaces;
    \item  $\eta: (\cO_F/n\cO_F)^2_S\to A[n]$ is an $\cO_F$-linear isomorphism from the constant group scheme $(\cO_F/n\cO_F)^2_S$ to $A[n]$.
\end{enumerate}
\end{definition}

  It is known that $\cH_{n}^{\mathfrak a}$ is represented by a Deligne-Mumford stack over $\mathbb F_p$ and when $n \geq 3$, it is even represented by a scheme (see \cite[Section 5 of Chapter 3]{Go}). From now on, we fix some $n\geq 3$ that is prime to $p$, a polarization $\mathfrak a$ and denote the representing scheme by $\cH$. We note that $\cH$ is smooth (see \cite[Theorem 2.1.2]{Pa}).
  
  We now set up some notation. Let $\abvar/\cH$ denote the universal Abelian surface over $\cH$, and let $\pdiv$ denote its $p$-divisible group. We have that $\End(\abvar) = \cO_F$, and as $p$ splits in $F$, we have that $\End(\pdiv) = \cO_F\otimes \ZZ_p = \ZZ_p\times \ZZ_p$. This implies that $\pdiv = \pdiv_1\times \pdiv_2$. At any point $x\in \cH$, let $\abvar_x$ and $\pdiv_{i,x}$ denote the pullback of $\abvar$ and $\pdiv_i$ to $x$. Note that both factors $\pdiv_i$ are one-dimensional and have height 2. 
  
  \begin{lemma}\label{lem:productdeformation}
  Let $x \in \cH(\Fpbar)$ and let the deformation space of $\pdiv_i$ be $\Spf \Fpbar[\![t_i]\!]$\footnote{Implicit here is that the deformation space of a $p$-divisible group having height 2 and dimension 1 is one-dimensional and smooth.}. Then $\widehat{\cH}_x$ is canonically isomorphic to $\Spf \Fpbar[\![t_1]\!] \times \Spf \Fpbar[\![t_2]\!]$.
  \end{lemma}
\begin{proof}
 The formal neighbourhood $\widehat{\cH}_x$ parameterizes formal deformations of $\abvar_x$ which admit an action of $\cO_F$ (compatible with the action of $\cO_F$ on $\abvar_x$). By the Serre-Tate lifting theorem, we have that this is the same as formal deformations of $\pdiv_x$ which admit an action of $\cO_F\otimes \ZZ_p$ -- but this is the same as pairs of $p$-divisible groups which lift the pair $(\pdiv_{1,x},\pdiv_{2,x})$. The lemma follows.
\end{proof}

Henceforth, we will refer to the decomposition above as the product structure on $\widehat{\cH}_x$.

\subsection{Partial Frobenius structure and coordinates}
For ease of exposition, we henceforth assume that $\mathfrak{P}_1$ (and therefore $\mathfrak{P_2}$) are trivial in the narrow class group of $\mathcal O_K$. Otherwise, the analysis below will go through identically, except we work with $p^a$ (and $\Frob_{p^a}$) in place of $p$ (and $\Frob_p$), where $a$ is the order of $\mathfrak{P}_i$ in the narrow class group. 

\begin{definition}\label{defn: r(x)}
We say a surface $X$ over $\mathbb{F}_p$ has partial Frobenius structure if there is a factorization
\[\Frob_p = \pi_1\pi_2,\]
where $\pi_i: X\rightarrow X$ are maps which satisfy the following condition: 

\noindent For any $q=p^r$ a power of $p$, locally around any $x\in X(\F_q)$, we can find coordinates $t_1,t_2$ such that:

\begin{equation}\label{eq: pfrob coord}
    (\pi_i^{r})^*(t_j) = \begin{cases}t_j^{q}&\text{ if } i=j,\\
    t_j &\text{ otherwise }.\end{cases}
\end{equation}
\end{definition}

A product of curves over $\F_p$ is one example of a variety that has a partial Frobenius structure (where the $\pi_1,\pi_2$ are induced from each factor independently), and there is an obvious choice of coordinates as in \eqref{eq: pfrob coord}. In particular, the product of modular curves will be a very pertinent example.

The key assumption that $p$ splits in $\cO_F$ results in a partial Frobenius structure on $\cH$. In this case, we may define the maps $\pi_i$ using our moduli interpretation. Suppose $S$ is an arbitrary $\mathbb F_p$-scheme and $\abvar/S$ is an abelian surface corresponding to a point $x \in \cH(S)$.   Define $\pi_1(x)$ to be the point $y$ corresponding to the abelian scheme $\abvar / G_1 $, where $G_1 \subset \pdiv_{1,x}$ is the kernel of Frobenius. According to Section 2.2 of \cite{Pa}, $\abvar \mapsto \abvar / G_1$ defines a morphism $\cH^{\mathfrak{a}}\to \cH^{\mathfrak{P}_1\mathfrak{a}}$. As $\mathfrak{P}_1$ is trivial in the narrow class group, we have indeed obtained a morphism $\pi_1: \cH^{\mathfrak{a}}\to \cH^{\mathfrak{a}}$. 

 We define $\pi_2(x)$ analogously. Clearly, $\Frob_p =\pi_1\pi_2$. Given this description, we see that for any $x\in \cH(\F_q)$, $\pi_1^r$ induces the $q^{\textrm{th}}$-power map on the deformation space of $\pdiv_{1,x}$ and leaves the deformation space of $\pdiv_{2,x}$ unchanged. We therefore have the following proposition: 


\begin{proposition}\label{prop: local co-ords}
$\cH$ has a partial Frobenius structure, with $\Frob = \pi_1\pi_2$. Furthermore, at every point $x\in \cH(\Fpbar)$, the coordinates $t_1,t_2$ are just the coordinates induced by the product structure as in Lemma \ref{lem:productdeformation}.
\end{proposition}
Let $x\in Z(\F_q)$ be a point, where $Z\subset \cH$ is the non-ordinary locus. Then according to the discussion above, the formal completion $\widehat{\Hilb}_{x} = \Spf \F_q[\![t_1,t_2]\!]$, where $t_i$ controls the deformation theory of $\pdiv_{x,i}$ and leaves $\pdiv_{x,i+1}$ constant (here, the indices are read modulo 2). Therefore, if $\pdiv_{x,1}$ is ordinary and $\pdiv_{x,2}$ is supersingular, the local equation of $Z$ at $x$ is just $t_1 = 0$; if $\pdiv_{2,x} $ is ordinary and $\pdiv_{1,x}$ is supersingular, then the local equation is $t_2 = 0$; and if $x$ itself is supersingular, then the local equation is $t_1t_2 = 0$.

  We also recall the global geometry of $Z \subset \Hilb$, following  \cite[\S4, Theorem 4.2]{BG} and relate it with the local geometry of $Z$ just discussed. The non-ordinary locus is a union of smooth irreducible curves which intersect transversally at supersingular points. There are two `types' of curves -- curves of type 1 and curves of type 2. Every non-supersingular point $x$ on curves of type 1 has the property that $\pdiv_{1,x}$ is ordinary while every non-supersingular point $y$ on curves of type 2 has the property that $\pdiv_{2,y}$ is ordinary. Curves of type 1 never intersect (and neither do curves of type 2) while every supersingular point has exactly one curve of type 1 and one curve of type 2 passing through it. Finally, the local equations of these curves are precisely as described in the paragraph just above.

We will use the following lemma in the proof of the theorem immediately after.

\begin{lemma}\label{lem: local to global}
Assume $K$ is a global function field, and $K_v$ is a completion of $K$ at some place $v$. Let $\pdiv_{/K}$ be an ordinary $p$-divisible group over $K$ with height 2 and dimension 1, and $\pdiv_{K_v}$ be its base change to $K_v$. If the connected-\'etale exact sequence for $\pdiv_{K_v}$ splits then the connected-\'etale exact sequence for $\mathcal G$ splits over $K$.
\end{lemma}
\begin{proof}
Jiang proves a more general result in \cite{ruofan}, but we include a different proof here for completeness. For any $n$, the connected-\'etale exact sequence of $\pdiv_{/K_v}[p^n]$ splits over some finite extension $L_n/K$.  We will prove that $L_n$ can be chosen to be a subfield of $K_v$. 

In order to show that we can choose $L_n \subset K_v$, we use the representability of $\Hom_{\text{gp sch}}(H,G)$ for any finite flat group schemes $H,G$. In general, $\Hom_{S}(X,Y)$ is representable if $X, Y$ are finite type schemes over $S$ and moreover, $X$ is flat, projective and $Y$ is quasi-projective over $S$ [\S 4c,  \cite{grothendieck1960techniques}] and the conditions for being a homomorphism of groups defines a closed subscheme. Splittings of the connected \'etale sequence
\[0 \to G_{n,\text{conn}} \to G_n \stackrel{\pi}{\to} G_{n,\text{\'et}} \to 0\]
can be identified with the closed subset of points of $f \in \Hom_{\text{gp sch}}(G_{n,\text{\'et}}, G_n)$ such that $f\circ\pi = \mathrm{id}$. In particular, they are parametrized by a scheme which we call $\Split/\Spec K$. 

We therefore have a splitting over $K_v, L_n$, i.e., $K_v,L_n$ valued points of $\Split$. We now consider $K_v,L_n \subset M$ for some algebraically closed field $M$. Let $\alpha \in \Split(M)$ be defined over some subfield $L$. Since any two splittings over $M$ differ by an automorphism of $G_{n,\text{conn}} \times G_{n,\text{\'et}}$ defined over $M$, and every one of these automorphisms is already defined over the ground field $K$, we have that every element of $\Split(M)$ is defined over $L$.  Therefore,
\[\Split(K_v) = \Split(M) = \Split(L_n) \implies \Split(K_v) = \Split(L_n\cap K_v).\]
Therefore, we can replace $L_n$ by $L_n \cap K_v$ to assume that $L_n \subset K_v$.

We now claim that $L_n/K$ is separable. Indeed, being a one-parameter function field, $K$ admits a unique degree $p$ inseparable extension necessarily containing all $p$-th roots of $K$, and which therefore contains the $p$-th root of a uniformizer of $K$ at $v$, which therefore can't be contained in $K_v$. It follows that $L_n/K$ is indeed separable. 

The splitting behaviour of the connected-\'etale exact sequence for any $p$-divisible group is insensitive to separable extensions, and therefore the connected-\'etale exact sequenc for $\pdiv[p^n]$ splits over $K$ for all $n$. The result follows.




\end{proof}

\begin{theorem}\label{lem: transversality}
Let $D\subset \cH$ be a generically ordinary, reduced, irreducible curve such that the generic abelian variety $A_{/K(D)}$ does not have an extra endomorphism\footnote{i.e., the ring of endomorphisms is $\cO_F$.}. We identify $\widehat{\cO}_{\cH,x}$ with $ \overline{\mathbb F}_p[\![t_1,t_2]\!]$ as in Lemma \ref{lem:productdeformation}. Let $\widehat{\cO}_{\cH,x}/(f_{D,x})$ be the formal completion of the local ring of $D$ at $x$. Then $t_i$ does not divide $f_{D,x}$ for any $i$.
\end{theorem}

\begin{remark}
As an analogous situation to the above lemma, consider $X=X(1)\times X(1)$, then locally around $x$, we may choose $t_i$ to be the corresponding coordinate of the $i$-th $X(1) \cong \mathbb{P}^1$. 
\end{remark}

\begin{proof}
Without loss of generality, we may replace $D$ with its normalization, and pull back the universal abelian surface to $D$. Therefore, we may assume that $D$ is smooth - however, the map $D\rightarrow \cH$ need no longer be an embedding. Nevertheless, given any point $x\in D$, we obtain a map of formal schemes $\psi: \widehat{D}_x \rightarrow \widehat{\cH}_x$ (we abuse notation by letting $x$ denote both the point of $D$ and its image in $\cH$). Let $f_{D,x} \in \widehat{\cO}_{\cH_{x}}$ denote the defining equation of the image of $\widehat{D}_x$. 

It suffices to prove that if $t_i$ divides $f_{D,x}$ for $i=1$ or $2$, then either $D$ is not generically ordinary, or $\abvar /K(D)$ has extra endomorphisms. Without loss of generality, assume $t_1 \mid f_{D,x}$. If $\pdiv_{x,1}$ is supersingular, then by the description of the non-ordinary locus following Proposition \ref{prop: local co-ords}, we see that $\abvar/D$ must be non-ordinary. Therefore, we suppose that $\pdiv_{x,1}$ is ordinary. 
The inclusion of the divisor corresponds to
\begin{align*}
    \widehat{D}_x \cong \operatorname{Spec} \overline{\mathbb F}_p[\![u]\!] &\to \widehat{\cH}_{d,x}\cong \operatorname{Spec} \widehat{\cO}_{\cH,x};\\
    t_1 &\stackrel{\psi}{\to} 0,\\
    \quad t_2 &\stackrel{\psi}{\to} \bar{t}_2 \in \overline{\mathbb F}_p[\![u]\!].
\end{align*}
Here, $\widehat{D}_x$ is the formal completion of $D$ at $x$. 
As in Lemma \ref{lem:productdeformation}, we know $\widehat{\cO}_{\cH,x}$ is the product of the deformation spaces of $\pdiv_{i,x}$. Notice that $\psi(t_1)=0$ implies that $\pdiv_{x,1}$ remains constant along $\widehat{D}_x$. In other words, $$\pdiv_{\widehat{D}_x}=(\pdiv_{x,1}\times \widehat{D}_x)\times  \pdiv_{\widehat{D}_x,2}.$$ 
The connected-\'etale exact sequence of $\pdiv_{x,1}$ splits since it is over a perfect field and hence, so does the connected-\'etale exact sequence of $\abvar_{\hat{D}_x}[p^\infty]_1$. By Lemma \ref{lem: local to global}, this implies that the connected-\'etale exact sequence for  $\pdiv_1$ already splits over $K(D)$. Consequently, $\End(\pdiv_{D,1}) =\Z_p\oplus \Z_p$, whence $\End(\pdiv_D)$ is strictly larger than $\Z_p\otimes \cO_F$. Applying the crystalline Tate conjecture for endomorphisms of abelian varieties (\cite[Theorem 2.6]{DJ}) implies $\abvar_D$ has endomorphism ring larger than $\cO_F$, as required.


\end{proof}

\section{Local intersection on varieties with a split Frobenius}\label{sec: local intersection bounds}

In this section, we assume that $X$ is a surface over $\mathbb F_p$ with a partial Frobenius structure. Throughout this section, we fix such a system of coordinates around each point as in \eqref{eq: pfrob coord}. As in Section \ref{sec: background}, Hilbert modular surfaces where $p = \mathfrak{P}_1\mathfrak{P}_2$ with $\mathfrak{P}_i$ trivial in the narrow class group of $\cO_F$ are examples of such surfaces. As noted in Section \ref{sec: background}, as long as $p$ splits completely in $\cO_F$, $\cH$ will admit a partial $p^a$-Frobenius structure where $a$ is the order of $\mathfrak{P}_1$ in the narrow class group. 

Note that the Frobenius (and hence the $\pi_i$) are universal homeomorphisms, i.e., they induce homeomorphisms on the underlying topological space. In particular, the preimage of any point under the $\pi_i$ is also exactly one point.

The results in this section will be purely local around a point $x \in X(\F_q)$. Note that $\pi_1$ need not fix $x$ but some power of it will so we may assume that $X = \Spec R$ with $R= \mathbb F_q[\![t_1,t_2]\!]$ and $x$ is the origin given by the vanishing of $t_1,t_2$. In local coordinates (a power of) $\pi_i$ corresponds to (a power of) Frobenius along the $t_i$ coordinates. For a point $x \in X$, we define $r(x)$ as the smallest value so that $\pi_1^{r(x)}(x) = \pi_2^{r(x)}(x) = x$.

Now, let $C,D \subset X$ be two curves and $x \in C \cap D$ and define $C_n = (\pi_1^{n})^{-1}(C)$ to be the pullback of $C$ under $\pi_1^n$, where $n$ satisfies $\pi_1^n(x) = x$. In this section, we prove results about the intersection numbers $C_n \cdot D$. Theorem \ref{thm: local intersection, moving curve} (1) owes its inspiration to \cite[Proposition 7.3]{CO06}.

\begin{theorem}\label{thm: local intersection, moving curve}
Let  $C = V(f), D = V(g)$ locally around $x \in \Hilb$ and $n = mr(x) \to \infty$ be part of a sequence of integers divisible by $r(x)$ and increasing without bound. Suppose moreover that $t_1,t_2 \nmid f$ and $t_1\nmid g$. 
\begin{enumerate}
    \item Suppose $t_2\nmid g$. Then, the local intersection number $(C_n\cdot D)_x$ is bounded as $n  \to \infty$.
    \item Suppose $t_2\mid g $. Then, the local intersection number $(C_n\cdot D)_x \to \infty$ as $n  \to \infty$.
\end{enumerate}
The analogous result holds with the roles of $\pi_1,\pi_2$ (and $t_1,t_2$) reversed.  
\end{theorem}
\begin{proof}

In order to compute intersection numbers, we can replace $D$ by its normalization and consider each component separately \cite[Example A.3.1]{fulton2013intersection}. We therefore have $\widehat\cO_{D,x} = \overline{\mathbb F}_p[\![u]\!]$ and we suppose that the morphism $D \to \Hilb$ is locally around $x$ given by
\begin{align*}
    \overline{\mathbb F}_p[\![t_1,t_2]\!] &\to \overline{\mathbb F}_p[\![u]\!]\\
    t_i &\to \alpha_i.
\end{align*}
For the first part of the theorem, $\alpha_1,\alpha_2 \neq 0$ since $t_1,t_2\nmid g$. We therefore define $\alpha_i = a_iu^{k_i}$ with $a_i$ a unit and $k_i \geq 1$. Since $x \in C$ and $r(x)|n$, we have $x \in C_n$. As $C$ is defined by $f(t_1,t_2)$, $C_n$ is defined by $f(t_1^{p^n},t_2)$ so that
\[(C_n.D)_x = \dim_{\overline{\mathbb F}_p}\frac{\overline{\mathbb F}_p[\![u]\!]}{f(\alpha_1^{p^n},\alpha_2)}.\]
By assumption, $t_1\nmid f$ so that we can write $f(t_1,t_2) = t_2^{e}\tilde{f}(t_2) + t_1h(t_1,t_2)$ for $e \geq 1$ and $\tilde{f}(t_2)$ a unit. Therefore,
\begin{align*}
    f(\alpha_1^{p^n},\alpha_2) &= \alpha_2^e\tilde{f}(\alpha_2) + \alpha_1^{p^n}h(\alpha_1^{p^n},\alpha_2)\\
    &= a_2^e u^{k_2e}\tilde{f}(a_2u^{k_2}) + a_1^{p^n}u^{p^nk_1}h(a_1^{p^n}u^{p^nk_1},a_2u^{k_2}).
\end{align*}
For $n$ large enough, the $u$-adic valuation of the second term is larger than the $u$-adic valuation of the first term since $\tilde{f}$ is a unit. Consequently (for $n$ large enough), $f(\alpha_1^{p^n},\alpha_2)$ is divisible by exactly $u^{ek_2}$ which proves that $(C_n.D)_x = ek_2$ which is independent of $n$.

In the second case, $\alpha_2 = 0$. Let $f(t_1,0) = t_1^dw(t_1)$ with $d\geq 1$ and $w(t_1)$ a unit. We then have (with $k_1$ as in the first part)
\begin{equation}\label{eq:hasse}
(C_n.D)_x = \dim_{\overline{\mathbb F}_p}\frac{\overline{\mathbb F}_p[\![u]\!]}{f(\alpha_1^{p^n},0)} = d k_1p^n
\end{equation}
which proves that $(C_n.D)_x \to \infty$ as $n \to \infty$.

\end{proof}
We have the following corollary.
\begin{corollary}\label{exactformula}
Let $C$ be as above, and suppose now that $Z_i \subset \widehat{\cH}_x$ are the formal curves defined by $t_i = 0$ for $i=1,2$. Then, $(C_n.Z_1)_x = (C.Z_1)_x$ and $(C_n.Z_2)_x = p^n(C.Z_2)_x $.
\end{corollary}
\begin{proof}
The first equality follows by inspection and the second from Equation \eqref{eq:hasse}.
\end{proof}

\section{Change of height under $p$-power isogenies}

In this section, we will describe some instances in which we can describe how the height of a generically ordinary proper curve $C \subset \cH$ changes under isogenies induced by the $\pi_i$. The idea is to use the fact that the Faltings height equals (up to a scaling factor of $\frac{1}{p-1}$) the intersection of $C$ with the non-ordinary locus. 

\subsection{The non-ordinary locus.}\label{sec:nonordinarylocus}

Let $Z$ be the non-ordinary locus. We recall the following description of $Z$ given in $\S 2$. The non-ordinary locus is a union of smooth irreducible curves which intersect transversally at supersingular points. There are two `types' of such curves -- curves of type 1 and curves of type 2 (coming from the local splitting of the Frobenius). Every non-supersingular point $x$ on curves of type 1 has the property that $\pdiv_{1,x}$ is ordinary while every non-supersingular point $y$ on curves of type 2 has the property that $\pdiv_{2,y}$ is ordinary. Curves of type 1 never intersect (and neither do curves of type 2) while every supersingular point has exactly one curve of type 1 and one curve of type 2 passing through it. Finally, the local equations of these curves are precisely as described in the paragraph following Proposition \ref{prop: local co-ords}.

For brevity, we write $Z = Z_1\cup Z_2$, where $Z_i$ is the union of curves of type $i$. The main result of this section is the following: 
\begin{theorem}\label{thm:nonordintersec}
Let $C/ \subset \cH$ denote a proper generically ordinary curve defined over $\F_{q'}$, and suppose that all the non-ordinary points of $C$ are contained in $C(\F_q)$ for some $q = p^{n_0}$. Then as $n$ goes to infinity, we have $((\pi_1^{n_0n})^{-1}(C)\cdot Z) =q^n (C. Z_2) + (C.Z_1)$, and $((\pi_2^{n_0n})^{-1}(C).Z) = q^n(C\cdot Z_1) + (C.Z_2)$. 
\end{theorem}

\begin{proof}
This follows directly from the description of the non-ordinary locus as well as Corollary \ref{exactformula}.
\end{proof}

We keep the notation in the previous result. Recall that we define the Faltings height of the abelian variety $\abvar_C$ to be the degree of the Hodge bundle $\omega$ restricted to $C$. As in \cite{MST1} and \cite{MST}, the class of the Hodge bundle can be expressed in terms of the non-ordinary locus. To be precise, we have that $Z$ is the vanishing locus of the Hasse invariant, which is a section of $\omega^{p-1}$ (for example, see \cite[Section 1.4]{boxer}). This, together with Theorem \ref{thm:nonordintersec} gives the following result.
\begin{theorem}\label{thm: change of heights formula}
The Faltings heights of $\abvar_{(\pi_1^{n_0n})^{-1}(C)}$ and $\abvar_{(\pi_2^{n_0n})^{-1}(C)}$ are \[h_F(\abvar_{(\pi_1^{n_0n})^{-1}(C)}) = \frac{q^n (C\cdot Z_2) + (C.Z_1)}{p-1} \text{ and } h_F(\abvar_{(\pi_2^{n_0n})^{-1}(C)}) = \frac{q^n (C. Z_1) + (C.Z_2)}{p-1}.\] In particular, if $C$ is also ample\footnote{This would imply that $(C.Z_1)$ and $(C.Z_2)$ are both positive. Even if $C$ is not ample, the curve $Z_1\cup Z_2$ is ample, and so either $(C.Z_1)$ or $(C.Z_2)$ must be positive, whence we obtain the same asymptotic formula for the Faltings heights of either $\abvar_{(\pi_1^{n_0n})^{-1}(C)}$ or $\abvar_{(\pi_2^{n_0n})^{-1}(C)}$.}, we have $h_F(\abvar_{(\pi_1^{n_0n})^{-1}(C)}) \asymp q^n \asymp h_F(\abvar_{(\pi_2^{n_0n})^{-1}(C)})$ as $n \to \infty$. 
\end{theorem}

\section{Just-likely intersections on Hilbert Modular Surfaces}

\begin{definition}
Given two proper curves $C,D \subset \Hilb$ defined over $\mathbb F_q$ parametrizing Abelian surfaces $\abvar_C,\abvar_D$, we let
\[\I(C,D) = \{(x,y) \in \Hilb^2 : \abvar_{C,x} \text{ is isogenous to } \abvar_{D,y} \text{ through a power of } \pi_i^n\}.\]
Moreover, let 
\[\IF(C,D) = \{(x,y) \in \Hilb^2 : \abvar_{C,x} \text{ is isogenous to } \abvar_{D,y} \text{ through a } p\text{-power isogeny }\}.\]
Note that $\I(C,D) \subset \IF(C,D)$.
\end{definition}

We use the results of Section \ref{sec: local intersection bounds} to prove:

\begin{theorem}\label{thm: main not contained in axes}
Let $C,D$ be as above and suppose that $\abvar_C,\abvar_D$ parametrized respectively by $C,D$ have no extra endomorphisms generically, are generically ordinary and suppose moreover that $D$ is ample.

Then $\I(C,D)$ has infinitely many points not contained in a finite union of axes of the form $ \bigcup_i \{x_i\}\times D \cup \bigcup_j C \times \{y_j\}$.
\end{theorem}

\begin{remark}
If one of the curves generically does have extra endomorphisms, then the theorem is still true by the results of \cite{MST}.
\end{remark}

\begin{theorem}\label{thm: main density}
Suppose that $C,D \subset \Hilb$ are two curves such that $\I(C,D)$ has infinitely many points not contained in a finite union of the axes as in the previous theorem. Then, $\IF(C,D)$ is dense in $C\times D$.
\end{theorem}

The strategy for the proof of Theorem \ref{thm: main not contained in axes} is as follows. We first prove that the local intersection numbers $d_n = ((\pi_1^{n})^{-1}(C). D)_x$ are bounded for any $x \in D$. Next, we prove that the global intersection numbers $((\pi_1^{n})^{-1}(C).D)$ are unbounded, by comparing these numbers with the quantities $((\pi_1^{n})^{-1}(C).Z)$, using Theorem \ref{thm:nonordintersec} and the ampleness of $D$. 
Finally, Theorem \ref{thm: main density} follows by a soft argument. Note that Theorem \ref{thm: main density} always holds for two curves which satisfy the conclusions of Theorem \ref{thm: main not contained in axes}. 

Throughout, we use the local coordinates $t_1,t_2$ of Lemma \ref{lem:productdeformation}. Let $C, D$ be as in Theorem \ref{thm: main not contained in axes}. We prove Theorem \ref{thm: main not contained in axes} through a sequence of lemmas. Recall that for a point $x \in X$, $r(x)$ is the smallest value so that $\pi_1^{r(x)}(x) = \pi_2^{r(x)}(x) = x$.

\begin{lemma}\label{lem: bdd local intersection number}
For any point $x \in C$, the intersection number of $(\pi_i^{n})^{-1}(C)$ with $D$ at $(\pi_i^{n})^{-1}(x) = x$ is bounded absolutely as $n$ ranges through the multiples of $r(x)$.
\end{lemma}
\begin{proof}
Around any point $x \in C,D$, the curves $C,D$ don't contain the divisors $t_i=0$ since otherwise they would either be generically non-ordinary or generically have extra endomorphisms (by Lemma \ref{lem: transversality}).

Therefore, part (1) of Theorem \ref{thm: local intersection, moving curve} applies and shows that the intersection number of $((\pi_i^{n})^{-1}(C)\cdot D)$ at $(\pi_1^{n})^{-1}(x)$ is absolutely bounded as $n = mr(x) \to \infty$.
\end{proof}

The next lemma deals with the global intersection number. We have that $(C.Z) = (C.Z_1) + (C.Z_2)$ is positive, so we assume without loss of generality that $(C.Z_2)$ is positive and define $C_n = (\pi_i^{n})^{-1}(C)$. 

\begin{lemma}\label{lem: global, Hasse}
The global intersection number $C_n\cdot Z \to \infty$ as $n \to \infty$ through the multiples of $r(x)$.
\end{lemma}
\begin{proof}
This follows immediately from Theorem \ref{thm:nonordintersec}.
\end{proof}


\begin{lemma}\label{lem: global, arbitrary}
Let $D$ now be an arbitrary ample, proper divisor. Then, the global intersection number $C_n\cdot D \to \infty$ as $n \to \infty$.
\end{lemma}
\begin{proof}
We will reduce to the case where our divisor is $Z$, the Hasse locus.

Since $D$ is ample, we can find a large enough $m$ such that $mD - Z$ is also ample. Therefore:
$$C_n.(mD - Z) \geq 0 \iff mC_n.D \geq C_n.Z$$
and since $C_n.Z \to \infty$ by the previous Lemma \ref{lem: global, Hasse}, $C_n.D$ is also unbounded. 
\end{proof}

We are now ready to prove the main results of this paper. 
\begin{proof}[Proof of Theorem \ref{thm: main not contained in axes}]
Let $S_n \subset C_n = (\pi_1^{n})^{-1}(C)$ be the set of points $x_{i,n}$ on $C_n$ isogenous to some point $y_{i,n}$ on $D$. Note that there is a unique $x_i \in C$ so that $\pi_1^n(x_{i,n}) = x_i$ since $\pi_1$ is a universal homeomorphism. We first show that $|S_n| \to \infty$ as $n \to \infty$. For contradiction, suppose that $|S_n| \leq N$. Then, for $r \gg 0$ and $n$ ranging through the multiples of the lcm of $r(x_1),\dots,r(x_N)$:
\[(C_{rn}\cdot D) = \sum_{i=1}^{N}(C_{rn}\cdot D)_{x_{i,n}} \leq NC_0\]
where the bound on the right hand side follows from Lemma \ref{lem: bdd local intersection number} and $C_0$ is some constant. On the other hand, the left hand side goes to infinity by Lemma \ref{lem: global, arbitrary} which provides us with our contradiction.

We have thus shown that there is an infinite set of points $x_1,x_2,\dots \subset C$ isogenous to some point on $D$. We would like to show that the corresponding points $y_1,y_2,\dots$ on $D$ also form an infinite set. Suppose otherwise for contradiction.

Thus, we can find an infinite subset $x_{i_1},x_{i_2},\dots$ isogenous to the same point $y$ on $D$ through the isogenies corresponding to $\pi_1^{n_1},\pi_2^{n_2},\dots$ with $n_i \to \infty$. That is:
\[\pi_1^{n_j}(y) = x_{i_j}.\]

On the other, the orbit of $\pi_1$ on any $\overline{\mathbb F}_q$ point of $\Hilb$ is finite since, in local coordinates, $\pi_1$ just corresponds to the Frobenius on one of the coordinates which certainly has a finite orbit. This forces us to identify some of the $x_{i_j}$ which is contradictory to our assumption that the $x_i$ are distinct.

We have thus shown that we can find two sequences $x_i,y_i \in C,D$ such that $x_i \neq x_j, y_i \neq y_j$ for $i\neq j$ and the $x_i$ are isogenous to the $y_i$ as required.
\end{proof}

We now prove Theorem \ref{thm: main density}, and therefore finish the proof of Theorem \ref{intromain}.

\begin{proof}[Proof of Theorem \ref{thm: main density}]
Suppose for contradiction that the closure of $\IF(C,D)$ is a proper subset of $C\times D \subset \Hilb^2$. Let $W$ be the complement of the axes inside the closure of $\I(C,D)$. By Theorem \ref{thm: main not contained in axes}, $W$ has positive dimension. We can therefore find an infinite sequence of points $(x_1,y_1),(x_2,y_2),\dots \in W$ such that the fields of definition of $y_i$ individually go to infinity.  We will prove that $N_i = \#\{y \in D : (x_i,y) \in W\} \to \infty$ as $i \to \infty$ which contradicts the fact that $W$ is a closed subset of $C\times D$ and therefore has finite degree projections onto the first factor.

Indeed, if $y_i$ has field of definition $\mathbb F_{q^{m_i}}$, then the size of the orbit of $y_i$ under $\Frob_{q}$ has size $m_i$ and moreover, each point in this orbit is $p$-power isogenous to $x_i$ and lies on $D$ (since $D$ is defined over $\mathbb F_q$). Since $W$ is defined over $\mathbb F_q$ too, it is fixed by any Frobenius and the entire orbit is contained inside $W$ proving that $N_i \geq m_i \to \infty$ as $i \to \infty$. 
\end{proof}



\bibliographystyle{alpha}

\bibliography{reference} 


\end{document}